\documentclass[12pt]{amsart}
\pdfoutput=1
\usepackage{microtype}
\usepackage[active]{srcltx}

\usepackage{hyperref}
\usepackage{calc,amssymb,amsthm,amsmath,amscd, eucal,ulem, mathtools}
\usepackage{phaistos}
\usepackage{alltt}
\RequirePackage[dvipsnames,usenames]{xcolor}

\input{mabliautoref.sty}
\input{xy}
\xyoption{all}
\input{kmacros3.sty}
\usepackage{tikz,calc}
\usepackage{tikz-cd}
\usetikzlibrary{external}

\usepackage[marvosym]{tikzsymbols}
\usepackage{amsfonts, mathrsfs}
\usepackage[cal=boondox, calscaled=1.05]{mathalfa}
\usepackage{calligra}
\usepackage{stmaryrd} 
\usepackage{dcpic, pictexwd}
\usepackage[left=1in,top=1in,right=1in,bottom=1in]{geometry}
\usepackage{bm}
\usepackage{verbatim}
\usepackage{upgreek}
\usepackage{todonotes}
\usepackage{mathtools}

\numberwithin{equation}{theorem}

\renewcommand{\:}{\colon}

\renewcommand{\m}{\mathfrak{m}}
\renewcommand{\n}{\mathfrak{n}}
\newcommand{\p}{\mathfrak{p}}
\newcommand{\kay}{\mathcal{k}}
\newcommand{\el}{\mathcal{l}}

\newcommand{\q}{\mathfrak{q}}

\DeclareMathOperator{\colim}{colim}
\DeclareMathOperator{\Ass}{Ass}

\let\Ass\relax
\DeclareMathOperator{\Ass}{\mathbf{Ass}}

\DeclareMathOperator{\ssHom}{ \sH\!\!\;\!\!\text{\calligra{\Large om}\,}}

\newcommand{\invrs}{^{-1}}

\usepackage{setspace}
\usepackage[shortlabels]{enumitem}
\usepackage{graphicx}
\usepackage[all,cmtip]{xy}

\usepackage{verbatim}

\renewcommand{\sC}{\mathcal{C}}

\renewcommand{\sF}{\mathcal{F}}
\renewcommand{\sG}{\mathcal{G}}
\renewcommand{\sH}{\mathcal{H}}

\renewcommand{\sO}{\mathcal{O}}

\usepackage{marvosym}

\begin{document}
\title{On Pristine Morphisms} 
\author[J.~Carvajal-Rojas]{Javier Carvajal-Rojas}
\address{Centro de Investigaci\'on en Matem\'aticas, A.C., Callej\'on Jalisco s/n, 36024 Col. Valenciana, Guanajuato, Gto, M\'exico}
\email{\href{mailto:javier.carvajal@cimat.mx}{javier.carvajal@cimat.mx}}
\author[A.~St\"abler]{Axel St\"abler}
\address{Universit\"at Leipzig\\ Mathematisches Institut\\ Augustusplatz 10\\
04109 Leipzig\\Germany} 
\email{\href{mailto:staebler@math.uni-leipzig.de}{staebler@math.uni-leipzig.de}}


\thanks{Carvajal-Rojas was partially supported by the grants ERC-STG \#804334, FWO \#G079218N, and SECIHTI \#CBF2023-2024-224, \#CF-2023-G-33, and \#CBF-2025-I-673. Stäbler was supported in part by SFB/TRR 45 Bonn-Essen-Mainz funded by DFG}


\begin{abstract}
We investigate flat morphisms of schemes of positive characteristic whose relative Frobenius is an isomorphism, which we call pristine. We show that these give rise to a natural Grothendieck topology that is fine tuned for the localization of Cartier modules.
\end{abstract}
\maketitle

\section{Introduction}

In algebraic geometry, one investigates geometric spaces cut out by polynomials with coefficients in a fixed base ring $\kay$. When $\kay$ is a field, it may have characteristic zero, like $\bC$ or $\bQ$, or positive characteristic $p>0$, like the finite fields $\bF_{p^e}$ and their extensions. These two settings produce geometries that are strikingly different, yet deeply interconnected.

In characteristic $p>0$, every such space $X$ comes with a canonical endomorphism $F = F_X \: X \to X$, its \emph{(absolute) Frobenius morphism}. This is in sharp contrast to characteristic zero, where nontrivial endomorphisms are actually quite rare. Roughly speaking, $F$ is defined by sending local coordinates to their $p$-th powers. This construction makes sense because the map $F \: r \mapsto r^p$
is a ring homomorphism precisely over rings of positive characteristic.

A large portion of algebraic geometry in positive characteristic reduces to analyzing Frobenius endomorphisms and their induced functorial actions. This leads to a wide array of invariants and notions defined via the Frobenius morphism. It is therefore crucial to understand how these objects transform under various maps $f\:Y \to X$. As one might expect, any natural description must involve the \emph{(relative) Frobenius morphism} $F_f \: X \to X^{(p)}$ of $f$, whose definition and basic properties will be given in \autoref{sec.pristineandregular}. In this work, we specialize to the situation where $f$ is flat and $F_f$ is an isomorphism; we call such morphisms \emph{pristine}. They form the class of maps along which Frobenius-theoretic structures should pass without obstruction.

It is nearly tautological that if $f\: Y \to X$ is a pristine morphism, then the Frobenius invariants of $X$ coincide with those of $Y$. However, to formulate this rigorously, we must first clarify the precise nature of pristine morphisms. We will find that such maps are intimately connected to formally étale morphisms and that they naturally give rise to an fpqc site, one that is finer than the pro-étale topology of Bhatt and Scholze \cite{BhattScholzeProetale}.

In fact, weakly étale morphisms are pristine, and, in turn, pristine morphisms are formally étale. We discuss this among other general properties of pristine maps in \autoref{sec.Pristinity}. Conversely, any formally étale morphism between locally noetherian schemes is necessarily pristine; see \autoref{fetaleprepristinefornoetherian}. This may be viewed as a Kunz-type theorem characterizing formal étaleness. The statement is non-trivial and was independently and simultaneously proved by Datta–Olander \cite{DataOlanderFIso} through methods distinct from ours. It is also worth noting that this equivalence was observed earlier by Daniel Fink in the $F$-finite setting \cite{FINKRelativeInverseLimitPerfection}, again using quite different techniques. 

Our strategy proceeds by establishing a Kunz-type theorem for formal unramification; see \autoref{NoetherianUnramifiedRelFClosedImmersion}. This depends on deep theorems already available in the literature, made possible by the recent work of Ma and Polstra \cite{polstramafsingularities}, who corrected substantial portions of the earlier results of Fogarty \cite{FogartyKahlerHilbert14thproblem} and Tyc \cite{TycDifferentialBasisAndp-basis}. We dedicate \autoref{sec.NoetherianPristinity} to the study of pristine morphisms in the noetherian setting.

From this vantage point, the guiding problem is to understand how Frobenius invariants behave locally along pristine maps. A prominent framework for expressing these invariants is that of \emph{Cartier modules}. In \autoref{sect.pristinepullbacks}, we develop a formalism for pulling back Cartier modules along pristine morphisms, and we will see that the notion of pristinity is fine tuned to make such pullbacks possible. Consequently, the pristine topology is tailored to the local analysis of Cartier modules. This, in turn, naturally points to a plethora of corollaries describing both the invariance and the local nature of $F$-singularities under pristine morphisms. We briefly indicate some basic instances, but in order to keep the present discussion focused on pristinity, a more extensive treatment will be pursued elsewhere.

\subsection*{Acknowledgements} 
 The authors would like to thank Manuel~Blickle, Rankeya~Datta, Anne~Fayolle, Daniel~Fink, Sri Iyengar, Peter~McDonald, Noah~Olander, Daniel~Smolkin, and Maciej~Zdanowicz for very useful discussions during the preparation of this article. 

\subsection*{General conventions} We will follow the following convention for the rest of this work. All schemes are defined over $\mathbb{F}_p$ for a fixed prime number $p$. Given a scheme $X$, we denote the $e$-th iterate of its (absolute) Frobenius endomorphism by $F^e=F^e_X\:X \to X$, and we use the customary short-hand notation $q \coloneqq p^e$. We say that $X$ is $F$-finite if $F$ is a finite morphism.

\section{Generalities on Frobenius}
\label{sec.pristineandregular}

In this preliminary section, we discuss the basics of the relative Frobenius of an arbitrary morphism of schemes. We aim to define the notions of \emph{$F$-flatness} and \emph{pristinity}, which we view; respectively, as the Frobenius-theoretic counterparts of regularity and \'etaleness. We further summarize some known results about the equivalence between Frobenius-theoretic properties and those defined by differentials.

Let $f \: X \to S$ be a morphism of schemes. The \emph{($e$-th relative) Frobenius morphism} of $f \: X \to S$ is defined as the dotted morphism in the following fibered product diagram:
\begin{align}\label{eqn.ReelativeFrobeniusDiagram}
\xymatrixcolsep{4pc}\xymatrixrowsep{3pc}\xymatrix{
X \ar@/_/[ddr]_-{f} \ar@/^/[drr]^-{F^e_X} \ar@{.>}[dr]|-{F^e_{f}}\\
&X^{(q)} \ar[d]_-{f^{(q)}} \ar[r]^-{G^e_f}  & X\ar[d]^-{f} \\
&S \ar[r]^-{F^e_S} &S}
\end{align}
We may also write $F^e_{X/S}=F^e_{f}$. For a map $f = \Spec \alpha \: \Spec A \to \Spec R$, we shall write
\[
F_{\alpha}^e = F^e_{A/R} \: A_{F^e_* R} \coloneqq A \otimes_R F^e_* R \to F^e_* A, \quad a \otimes F^e_*r\mapsto F^e_* a^qr
\]
for the corresponding homomorphism $\alpha \: R \to A$. We may also write $F^e_{\alpha,*} A$ to denote the $A_{F^e_* R}$-module obtained from viewing $F^e_* A$ as an $A_{F^e_* R}$-module.

The construction of relative Frobenius is functorial, as explained in \cite[\href{https://stacks.math.columbia.edu/tag/0CCA}{Tag 0CCA}]{stacks-project}. For further details on relative Frobenius morphisms, see \cite[\href{https://stacks.math.columbia.edu/tag/0CC6}{Tag 0CC6}]{stacks-project}. As an example, we showcase the following result, which we will use repeatedly. 

\begin{proposition}[{\cite[\href{https://stacks.math.columbia.edu/tag/0CCB}{Tag 0CCB}]{stacks-project}}] \label{lem.RelativveFrobeniiUniversalHomeos}
With notation as above, $F^e_{f}$ is an integral universal homeomorphism that induces purely inseparable residue field extensions.
\end{proposition}

\begin{remark}[Stability under base change and naturality]
\label{rem.relFrobstableunderbasechange}
The map $F^e_{f} \: X \to X^{(q)}$ is a morphism of $S$-schemes and its base change along a morphism $S' \to S$ is $F^e_{f'} \: X' \to X'^{(q)}$, where $f' \coloneqq f \times_S S' \: X' \to S'$. Also, the formation of relative Frobenius is natural in the following sense. Consider a commutative triangle
\[
\xymatrix{
Y \ar[dr]_-{g} \ar[rr]^-{h} && X \ar[dl]^-{f} \\
 &S&
}
\]
Then,
\begin{equation} \label{eqn.NaturalityFrob}
F^e_{g} =   \bigl( F^e_{f} \times_X Y \bigr) \circ F^e_{h}
\end{equation}
where $F^e_{f} \times_X Y$ is the base change of $F^e_{f}$ along $Y \to X$ as a morphism of $X$-schemes.
\end{remark}

    The base change of $F^e_f\: X \to X^{(q)}$ along $F^{e'}_S \: S \to S$ is the map
    \[
    F^e_{f^{(q')}} \: X^{(q')} \to X^{(qq')}.
    \]
    Its composition with $F_f^{(q')}\: X \to X^{(q')}$ is $F_f^{e+e'} \: X \to X^{(qq')}$. In particular, for a property $\mathcal{P}$ of morphisms that is stable under base change and closed under composition (finiteness, flatness, being an isomorphism), if $F^e_f$ satisfies $\mathcal{P}$, then so does $F^{ne}_f$ for all $n\geq 1$. Likewise, if $\mathcal{P}$ is further inherited to right factors in compositions (e.g., finiteness conditions), then we have that $F_f^e$ satisfying $\mathcal{P}$ for all $e>0$ or some $e>1$ are equivalent notions. This takes us to the following notion.

    \begin{definition}[{$F$-finiteness \cf \cite{HashimotoFfinitenessAndItsdescent}}]
We say that $f\: X \to S$ is \emph{$F$-finite} if $F^e_f$ is finite (and necessarily surjective) for some (equivalently all) $e \geq 1$. Similarly, $f\: X \to S$ is \emph{fiberwise $F$-finite} if the fibers $X_s$ are $F$-finite schemes for all $s \in S$. 
\end{definition}

\begin{remark}
     Observe that $X$ is $F$-finite if and only if $f$ and $S$ are $F$-finite. Also, $f$ is fiberwise $F$-finite if it is $F$-finite. For more on $F$-finiteness, see \cite{HashimotoFfinitenessAndItsdescent}
\end{remark}

\begin{remark}[Kählerianity vs $F$-finiteness]
   Recall that a morphism of schemes $f\: X \to S$ is said to be \emph{kählerian} if $\Omega_f$ is (a quasi-coherent sheaf) of \emph{finite generation} \cite{KunzKahlerDifferentials,SeydiRestucciaBonanzingaCangemiUtanoSurLaTheorieIII}.\footnote{By \emph{quasi-coherent sheaf of finite generation}, we mean that $\Omega_f(U)$ is a finitely generated $\sO_X(U)$-module for all affine open subsets $U\subset X$. This is the condition (iii) in \cite[Chapter 5, Proposition 1.11]{LIUAGAC}, where it is stated to be weaker than coherency but equivalent to it if $X$ is locally noetherian.}  This includes the class of morphisms essentially of finite type. Since $ \Omega_f = \Omega_{F_f}$ (\cite[\href{https://stacks.math.columbia.edu/tag/0CCC}{Tag 0CCC}]{stacks-project}), this would further include the class of $F$-finite morphisms. Moreover, $\Omega_f$ is a coherent sheaf provided that $f$ is $F$-finite and $X$ is locally noetherian. Conversely, a kählerian morphism is $F$-finite provided that $F_f$ is a morphism between locally noetherian schemes. This is a statement claimed by Fogarty \cite[Proposition 1]{FogartyKahlerHilbert14thproblem} only assuming that $X$ is locally noetherian (i.e., without assuming locally noetherianity for the scheme-theoretic image of $F_f$). Fogarty's proof, however, is flawed. Fortunately, a valid proof has been recently worked out by Ma--Polstra in \cite[Chapter 11]{polstramafsingularities}, where the reader can find more details.\end{remark}

Next, we want to show that $F_f$ is an isomorphism if and only if $F_f^e$ is an isomorphism for some $e > 0$. To argue this and see what happens for flatness, consider the following commutative diagram, which illustrates how the relative Frobenius behaves under iteration.
    \[
   \xymatrixcolsep{5pc} \xymatrix{
    \ddots  \ar[rd]^-{F_X} &  &  &  &  \\
    \cdots \ar[r]^-{G_f} & X \ar[rd]^-{F_X} \ar[d]^-{F_f} &  &  &  \\
    \cdots \ar[r] ^-{G_{f^{(p)}}} & X^{(p)} \ar[rd]^-{F_{X^{(p)}}} \ar[r] \ar[d]^-{F_{f^{(p)}}} \ar[r]^-{G_f} & X \ar[rd]^-{F_X} \ar[d]^-{F_f} &  &  \\
     \cdots \ar[r]^-{G_{f^{(p^2)}}} & X^{(p^2)} \ar[rd]^-{F_{X^{(p^2)}}}\ar[r] \ar[d]^-{F_{f^{(p^2)}}} \ar[r]^-{G_{f^{(p)}}} & X^{(p)} \ar[rd]^-{F_{X^{(p)}}} \ar[r] \ar[d]^-{F_{f^{(p)}}} \ar[r]^-{G_f} & X \ar[rd]^-{F_X} \ar[d]^-{F_f} &  \\
    \cdots \ar[r]^-{G_{f^{(p^3)}}} & X^{(p^3)} \ar[r] \ar[d]^-{f^{(p^3)}} \ar[r]^-{G_{f^{(p^2)}}} & X^{(p^2)} \ar[r] \ar[d]^-{f^{(p^2 )}} \ar[r]^-{G_{f^{(p)}}} & X^{(p)} \ar[r] \ar[d]^-{f^{(p)}} \ar[r]^-{G_f} & X \ar[d]^-{f}  \\
   \cdots \ar[r]^-{F_S} & S \ar[r]^-{F_S} & S \ar[r]^-{F_S} & S \ar[r]^-{F_S} & S
    }
    \]
    In this diagram, all the squares are cartesian. This diagram has the salient feature of showing that $F_f$ and $G_f$ are ``mutual inverses up to Frobenius'' as coined in \cite[Chapter 11]{polstramafsingularities}. This will play an important role later in the proof of \autoref{NoetherianUnramifiedRelFClosedImmersion}. As for now, observe that we also obtain that
    \[
    F^e_f = F_{f^{(p^{e-1})}}\circ \cdots \circ F_{f^{(p^2)}}\circ F_{f^{(p)}} \circ F_f
    \]
    and likewise
    \[
    G^e_f = G_{f^{(p^{e-1})}}\circ \cdots \circ G_{f^{(p^2)}}\circ G_{f^{(p)}} \circ G_f.
    \]
    By taking (inverse/projective) limits as $e \rightarrow \infty$, we conclude that
    \[
    F_f^{\infty} = \lim_e F_f^e
    \]
    where $F_f^{\infty}$ is the \emph{perfection} of $f$, which is defined by the following cartesian square
    \[
\xymatrixcolsep{6pc}\xymatrix{
X_{\mathrm{perf}} \ar@/_/[ddr]_-{f_{\mathrm{perf}}} \ar@/^/[drr]^-{F^{\infty}_X} \ar@{.>}[dr]|-{F^{\infty}_{f}}\\
&X_{S_{\mathrm{perf}}} \ar[d]_-{f_{S_{\mathrm{perf}}}} \ar[r]_-{G^{\infty}_f = \lim_e G_f^e}  & X\ar[d]^-{f} \\
&S_{\mathrm{perf}} \ar[r]^-{F^{\infty}_S} &S}
\]
Furthermore, this implies that the map
\[
F_{f^{(p^{\infty})}} \coloneqq \lim_e F_{f^{(q)}} = (F_f)_{X_{S_{\mathrm{perf}}}} = F_{X_{S_{\mathrm{perf}}}} \: X_{S_{\mathrm{perf}}} \to X_{S_{\mathrm{perf}}} 
\]
fits into the following commutative diagram
\[
\xymatrix{
X_{\mathrm{perf}}\ar[rd]_-{F_f^{\infty}} \ar[r]^-{F_f^{\infty}} & X_{S_\mathrm{perf}} \ar[d]^-{F_{f^{(p^{\infty})}}} \\
& X_{S_\mathrm{perf}}
}
\]
There is one more factorization of $F_{f^{(p^{\infty})}}$ that we obtain from this, namely
\[
\xymatrixcolsep{5pc}\xymatrix{
X_{\mathrm{perf}}\ar[rd]_-{F_f^{\infty}} \ar[r]^-{(F^e_{f})_{\mathrm{perf}}} & X^{(q)}_\mathrm{perf} \ar[d]^-{F^{\infty}_{f^{(q)}}} \\
& X_{S_\mathrm{perf}}
}
\]
whose limit as $e \rightarrow \infty$ yields
\begin{equation} \label{eqn.DiagramRelativePerfection}
    \xymatrixcolsep{5pc}\xymatrix{
X_{\mathrm{perf}}\ar[rd]_-{F_f^{\infty}} \ar[r]^-{(F^{\infty}_f)_{\mathrm{perf}}} & (X_{S_\mathrm{perf}})_{\mathrm{perf}} \ar[d]^-{F^{\infty}_{X_{S_{\mathrm{perf}}}} = \lim_e   F^{\infty}_{f^{(q)}}} \\
& X_{S_\mathrm{perf}}
}
\end{equation}

With the above in place, we readily conclude the following.

\begin{proposition} \label{pro.FrobBeingAnIso}
    With notation as above, consider the following statements:
    \begin{enumerate}
        \item $F_f$ is an isomorphism.
        \item $F^e_f$ is an isomorphism for all $e>0$.
        \item $F^e_f$ is an isomorphism for some $e>0$.
        \item $F_f^{\infty}$ is an isomorphism.
        \item $X_{S_{\mathrm{perf}}}$ is perfect.
    \end{enumerate}
    Then, the following implications hold:
    \[
    (a) \Longleftrightarrow (b)  \Longleftrightarrow (c)  \Longrightarrow (d) \Longrightarrow (e) .
    \]
    Moreover, $(e)\Longrightarrow (a)$ if $F_S \: S \to S$ is flat.
\end{proposition}
\begin{proof}
    It remains to explain why (c) implies (a) as well as why (e) implies (a) if $F_S$ is flat. For the former, suppose that $F_f^e$ is an isomorphism. Since $F_f^e = F^{e-1}_{f^{(p)}} \circ F_f$, we conclude that $F_f$ admits a left inverse. Since $F_f$ is affine, this implies that $F_f$ is a closed immersion. Hence, so is its base change $F^{e-1}_{f^{(p)}}$. From this, we see that if $F_f$ is not an isomorphism, neither is $F_f^e$. 

    Finally, we explain why (e) implies (a) if $F_S$ is flat.  This follows from our observation above that the base change of $F_f$ along $X^{(p)} \from X_{S_{\mathrm{perf}}}$ is none other than the Frobenius map of $X_{S_{\mathrm{perf}}}$. If $F_S$ is flat, then $X^{(p)} \from X_{S_{\mathrm{perf}}}$ is an fpqc covering, and one simply descends the property of being an isomorphism.
\end{proof}

\begin{remark} \label{rem.RelativePerfectionBeingIsoDoesNotImpolyFrobIsISo}
It is not true that in \autoref{pro.FrobBeingAnIso}
(e) nor (d) imply (a) without assuming that $F_S$ is flat. Consider, for example, the homomorphism $\bF_p[x]/x^2 \to \bF_p$ sending $x$ to $0$. Also see \cite[Remark 2.2]{DumitrescuRegularityFiniteFlatDimension}.
\end{remark}

As for flatness, we can say the following.

\begin{proposition} \label{pro.FrobBEingFlat}
    With notation as above, consider the following statements:
    \begin{enumerate}
        \item $F_f$ is flat.
        \item $F^e_f$ is flat for all $e>0$.
        \item $F^e_f$ is flat for some $e>0$.
        \item $F_f^{\infty}$ is flat.
        \item $F_{X_{S_{\mathrm{perf}}}}$ is flat.
    \end{enumerate}
    Then, the following implications hold:
    \[
    (a) \Longleftrightarrow (b)  \Longrightarrow (c)  \Longrightarrow (d) \Longrightarrow (e) .
    \]
    Moreover, $(e)\Longrightarrow (a)$ if $F_S \: S \to S$ is flat.
\end{proposition}
\begin{proof}
The proof of (e) implies (a) if $F_S$ is flat is by fpqc descent as in \autoref{pro.FrobBeingAnIso}, but this time for flatness. Also see \cite[Lemma 2.3]{DumitrescuRegularityFiniteFlatDimension}. To see why (d) implies (e), use the diagram \autoref{eqn.DiagramRelativePerfection} together with the general fact that if a composition $h\circ g$ is flat and $g$ is \emph{faithfully} flat, then $h$ is flat.
\end{proof}

\begin{remark}
    The authors do not know whether $F_f^e$ being flat for some $e$ implies that $F_f$ is flat (without assuming that $F_S$ is flat). Further, the same example in \autoref{rem.RelativePerfectionBeingIsoDoesNotImpolyFrobIsISo} shows that $F_f^{\infty}$ being flat does not imply that $F_f$ is flat.
\end{remark}

The above leads to the following notion.

\begin{definition}[Pristinity and $F$-flatness] A morphism $f\: X \to S$ is \emph{pre-pristine} if $F^e_f$ is an isomorphism for some (equivalently all) $e \geq 1$. We say that $f$ is \emph{pristine} if it is flat and pre-pristine.\footnote{Naturally, the term to use is \emph{perfect}. However, it is employed for a different notion; see \cite[\href{http://stacks.math.columbia.edu/tag/06BY}{Tag 06BY}]{stacks-project}. Thus we decided to use a synonym. To the best of the authors' knowledge, pristine morphisms have not been named before other than saying that $X/S$ is relatively perfect; see \cite{FINKRelativeInverseLimitPerfection}.} We say that $f$ is \emph{$F$-flat} if it is flat and $F_f^e$ is flat for all $e \geq 1$ (equivalently for $e=1$).
\end{definition}

\begin{remark} \label{rem.WeakiningsOnF_fForPrisitinity}
Recall that $F_f$ is a universal homeomorphism and thus it is surjective (see \autoref{lem.RelativveFrobeniiUniversalHomeos}). Then, it is an isomorphism if and only if it is an open immersion. On the other hand, flat closed immersions between locally noetherian schemes are open immersions; see \cite[I, Proposition 3.10]{MilneEtaleCohomology} or \cite[\href{https://stacks.math.columbia.edu/tag/05KK}{Lemma 05KK}]{stacks-project}. Thus, if $X^{(p)}$ is locally noetherian, an $F$-flat morphism $f$ is pristine if and only if $F_f$ is a closed immersion. This will be important in \autoref{sec.NoetherianPristinity}.
\end{remark}

\begin{example}[Étaleness vs pristinity]
    \'Etale morphisms are pristine; see \cite[\href{http://stacks.math.columbia.edu/tag/0EBS}{Tag 0EBS}]{stacks-project}. We will see below (see \autoref{rem.LTrivialImpliesFormallyEtale}) that \'etale morphisms are precisely the pristine morphisms of finite type between locally noetherian schemes.
\end{example}

\begin{example}[Pristine field extensions]
    A field extension $ K \to L$ is pristine if and only if it is separable with the empty set as a $p$-basis. Indeed, by \cite[\href{https://stacks.math.columbia.edu/tag/0322}{Proposition 0322}]{stacks-project}, $F_{L/K}$ is injective if and only if $L/K$ is separable (i.e., formally smooth). On the other hand, $F_{L/K}$ is surjective if and only if $K[L^p] = L$ which, by definition, is to say that $\emptyset$ is a $p$-basis (\cite[\href{https://stacks.math.columbia.edu/tag/07P2}{Lemma 07P2}]{stacks-project}).
Since the image of any $p$-basis under the universal derivation $\mathrm{d}\:L \to \Omega_{L/K}$ is a basis of $\Omega_{L/K}$ we conclude that $L/K$ is pristine if and only if it is formally \'etale (see \cite[\href{https://stacks.math.columbia.edu/tag/00UO}{Lemma 00UO}]{stacks-project}, \cite[\href{https://stacks.math.columbia.edu/tag/0322}{Proposition 0322}]{stacks-project}). In particular, for any field $\kay$, the extension $\kay \subset \kay(t^{1/p^\infty})$ is pristine. It is worth noting that, more generally, we have that $F_{L/K}$ is injective (resp. surjective) if and only if $\Omega_{K/\bF_p} \otimes_K L \to \Omega_{L/\bF_p}$ is injective (resp. surjective), which means that $L/K$ is formally smooth (resp. formally unramified). We intend to explore to what extent this principle holds for arbitrary morphisms of schemes.
\end{example}

\begin{example}[Local completions]
\label{pro.completionpristineFfinite}
Let $(R, \fram,\kay)$ be a noetherian $F$-finite local ring. The completion $\theta\: R \to \hat{R}$ is pristine. Indeed, since $\theta$ is flat (as $R$ is noetherian), we only need to show that $F_{\theta}$ is an isomorphism. Since $R$ is also $F$-finite, we have that $\hat{R}_{F_*R}$ is the $\fram^{[p]}$-adic completion of $R$, which coincides with the $\fram$-adic completion as $R$ is noetherian.
\end{example}

\subsection{Regularity vs $F$-flatness}

 We recall the following theorem by Radu--Andr\'e (also known as the relative Kunz theorem). See \cite[\href{https://stacks.math.columbia.edu/tag/07R6}{Tag 07R6}]{stacks-project} for details on regular morphisms.

\begin{theorem}[{\cite{RaduRelativeKunz,AndreRelativeKunz}, \cf \cite[\S 2]{EnescuOnTheBehaviorOfFrationalRingsUnderFlatBaseChange}}]
\label{theo.andreraduthm}
Let $f\: X \to S$ be a morphism between locally noetherian schemes. Then, $f$ is regular if and only if it is $F$-flat.
\end{theorem}
\begin{proof}
See \cite[Section 10]{polstramafsingularities} for a proof avoiding André--Quillen homology.
\end{proof}

Arguably, the most important ingredient in the proof of Radu--Andr\'e's theorem is the following theorem. For instance, from it, the $F$-flatness of regular morphisms follows from \emph{crit\'ere de platitude par fibres}; see \cite[\href{https://stacks.math.columbia.edu/tag/039A}{Section 039A}]{stacks-project}, \cf \cite[IV 5.9]{GrothendieckSGA}. This is another key ingredient in \autoref{sec.NoetherianPristinity}.

\begin{theorem}
\label{theo.raduandrecore}
Let $f \: X \to S$ be a regular morphism of locally noetherian schemes. Then, $X^{(q)}$ is locally noetherian.
\end{theorem}
\begin{proof}
See \cite[Th\'eor\`eme 25 and Demonstration 37]{AndreAutreDemosntrationRelKunz} or \cite[Section 10]{polstramafsingularities}.
\end{proof}

\begin{remark}
Since regular morphisms are flat, pre-pristine morphisms are pristine provided that both source and target schemes are locally noetherian. In \autoref{ex.PrePristineNotPristine} below, we will encounter a (non-noetherian) example of a morphism that is pre-pristine but not pristine. In particular, (local) noetherianity is a fundamental hypothesis in Andr\'e's proof that flatness of relative Frobenius implies flatness of the given morphism.
\end{remark}

\subsubsection{Beyond the noetherian case}
\label{rem.formallysmoothcharacterization} 
Let $f\: X \to S$ be a morphism between locally noetherian schemes. It is possible to characterize its regularity using differentials or, more precisely, using its cotangent complex $\bL_f$. Namely, $f$ is regular if and only if $H_1(\bL_f) = 0$ and $H_0(\bL_f)\coloneqq \Omega_f$ is flat; see \cite[Proposition 5.6.2]{Majadas2010}. However, this last differential condition makes sense beyond the noetherian case, and we may refer to it as \emph{homological regularity}. In particular, a general morphism $g$ is formally smooth (resp. formally étale) if and only if it is homologically regular and $\Omega_g$ is locally projective (resp. trivial); see \cite[Theorem 2.3.1, Proposition 5.6.2]{Majadas2010}, \cf \cite[\href{https://stacks.math.columbia.edu/tag/06B5}{Tag 06B5}]{stacks-project}. In this way, for an $F$-finite map, homological regularity and formal smoothness are the same notion. It is worth noting that recently Majadas--Alvite--Barral \cite{MajadasAlviteBarralFormalRegularity} introduced a notion of \emph{formal regularity} for general morphisms that coincides with regularity in the noetherian case. They prove that $F$-flat maps are formally regular. However, this notion of formal regularity is weaker than ours of homological regularity and it does not seem to imply that $\Omega_f$ is flat, which we believe is essential. Thus, we may ask:
\begin{question}
    Are $F$-flat morphisms homologically regular?
\end{question}
Since $F_X$ is flat for an $F$-flat map $f\: X \to S$, fpqc descent and \autoref{pro.Pre-pristineHavetrivialCotangentComplex} imply that this is the same as asking whether $F$-flat maps satisfy that $H_2(\bL_{F_f^{\infty}}) = 0$ and $H_1(\bL_{F_f^{\infty}})$ is flat. By the Radu--André theorem, this has a positive answer when $X$ and $S$ are both locally noetherian.

\section{On Pristinity} \label{sec.Pristinity}

In this section, we collect some general, basic observations on the class of pristine morphisms.

\begin{proposition}
\label{prepristineproperties}
Let $S$ be a scheme. The class of (pre-)pristine $S$-schemes is:
\begin{enumerate}[label=(\alph*)]
\item stable under any base change $S'\to S$,
\item closed under composition, further, an $S$-morphism of pre-pristine $S$-schemes is pre-pristine, and
\item closed under product, i.e., if $X/S$ and $Y/S$ are pristine $S$-schemes, then so is the $S$-scheme $X \times_S Y$.
\end{enumerate}

\end{proposition}
\begin{proof}
(a) follows since flatness and relative Frobenius are stable under base change (see \autoref{rem.relFrobstableunderbasechange}). The naturality of relative Frobenius expressed in \autoref{eqn.NaturalityFrob} implies (b). Finally, (c) is a direct consequence of (a) and (b). 
\end{proof}

\begin{proposition}
\label{pro.pristinezariskilocal}
Being (pre)-pristine is (Zariski-)local on both the source and the target (see \cite[\href{https://stacks.math.columbia.edu/tag/036G}{Tag 036G}]{stacks-project} and \cite[\href{https://stacks.math.columbia.edu/tag/02KO}{Tag 02KO}]{stacks-project}). That is, $f\: X \to S$ is (pre-)pristine if and only if for every open subschemes $U \subset X, V \subset S$ with $V \subset f(U)$ the restriction $f\: U \to V$ is (pre-)pristine.
\end{proposition}
\begin{proof}
Flatness is Zariski-local on both the source and the target. Let $f\: X \to S$ be pre-pristine. Let $U \subset S$ be any open set and $f_U\: f^{-1}(U) \to U$ be the pullback of $f$ to $U$. By \cite[\href{https://stacks.math.columbia.edu/tag/01JR}{Tag 01JR}]{stacks-project}, $F_{f_U}$ is the restriction of $F_f$ and thus is also an isomorphism. The converse follows since $F_f^{-1}$ is obtained by gluing the $F_{f_U}^{-1}$ for any open covering. Hence, pre-pristinity is local on the target. Pre-pristinity is also local on the source. If $U \subset X$ is any open set, then the relative Frobenius $U \to U^{(q)}$ is the base change of $F_f$ by the open immersion $U \to X$ and thus pre-pristine by (a) in \autoref{prepristineproperties}
\end{proof}

\begin{remark}
The above propositions can be reinterpreted by stating that fpqc pristine morphisms can be used as a class of coverings to construct both small and big sites. 
\end{remark}

As an immediate consequence of \autoref{pro.pristinezariskilocal}, we obtain the following.

\begin{corollary}
\label{le.pristinelocalonsource}
Let $f \: X \to S$ be a morphism of schemes. The following statements are equivalent.
\begin{enumerate}[label=(\alph*)]
\item{The morphism $f$ is (pre-)pristine.}
\item{For every affine open $U \subset X$, $V \subset S$ with $V \subset f(U)$ the homomorphism $\sO_S(V) \to \sO_X(U)$ is (pre-)pristine.}
\item{There exists an open covering $S=\bigcup_{j \in J} V_j$ and open coverings $f^{-1}(V_j)= \bigcup_{i\in I_j } U_{ij}$ such that the restriction $f|_{U_{ij}}\: U_{ij} \to V_j$ is (pre-)pristine for all $i \in I_j$ and all $j \in J$.}
\item{There exists an affine open covering $S=\bigcup_{j \in J} V_j$ and affine open coverings $f^{-1}(V_j)= \bigcup_{i\in I_j } U_{ij}$ such that $ \sO_S(V_j) \to \sO_X(U_{ij})$ is (pre-)pristine for all $i \in I_j$ and all $j \in J$.}
\end{enumerate}
Moreover, if $f$ is (pre-)pristine then for any open subschemes $U \subset X$, $V \subset S$ with $V \subset f(U)$ the restriction $f|_U\: U \to V$ is (pre-)pristine.
\end{corollary}

\begin{proposition}[fpqc descent for pristinity] \label{FlatDescentForPristinity}
Consider a cartesian square
\[
\xymatrix{
X' \ar[r]^-{h} \ar[d]_-{f'} & X \ar[d]^-{f} \\
S' \ar[r]^-{g} & S
}
\]
If $g$ is faithfully flat and quasi-compact, then $f$ is (pre-)pristine if and only if so is $f'$. 
\end{proposition}
\begin{proof}
By \cite[Proposition 2.5.1, Proposition 2.7.1]{EGA_IV_II}, faithfully flat morphisms and isomorphisms satisfy faithfully flat descent.
Thus we only have to check that $F_{f'}$ is an isomorphism if and only if $F_f$ is one. As noted in \autoref{rem.relFrobstableunderbasechange}, $F_f \times \id_{S'} = F_{f'}$ so that the result follows.
\end{proof}

\begin{proposition}[fppf descent for pristinity]
Consider a cartesian square
\[
\xymatrix{
X' \ar[r]^-{h} \ar[d]_-{f'} & X \ar[d]^-{f} \\
S' \ar[r]^-{g} & S
}
\]
If $g$ is faithfully flat and locally of finite presentation, then $f$ is (pre-)pristine if and only if so is $f'$. 
\end{proposition}
\begin{proof}
This follows along the same lines as \autoref{FlatDescentForPristinity} (\cf \cite[Proposition 1.15]{VistoliGrothendieckTopologies}).
\end{proof}

The following is a generalization of \autoref {prepristineproperties} (c). For this, we will introduce the following setup. Let $\mathsf{C}/S$ be a full subcategory of $\mathsf{Sch}/S$---the category of $S$-schemes. We will denote the full subcategory of pristine $S$-schemes in $\mathsf{C}/S$ by $\mathsf{Pri^C}/S$.
\begin{proposition} \label{PristinityAndLimits}
Let $I$ be a small filtered category and consider a diagram $D \: I \to \mathsf{Pri^C}/S$. If the limit $\lim_{i \in I}{D(i)}$ of this diagram exists in $\mathsf{C}/S$, then it exists in $\mathsf{Pri^C}/S$. Thus, the canonical morphisms $\lim_{i \in I}{D(i)} \to D(i)$ are pristine. In particular, ind-\'etale homomorphisms of rings are pristine.
\end{proposition}
\begin{proof}
This is a formal consequence of the abstract nonsense ``limits commute with limits,'' for fiber products are limits. More precisely, the relative Frobenius of $ \lambda \: \lim_{i \in I}{D(i)} \to S$ is the limit of the relative Frobenius maps $F^e_{D(i)/S} : D(i) \to D(i)^{(q)}$ as the following commutative diagram makes precise
\[
\xymatrixcolsep{3.5pc}\xymatrix{
\lim_{i \in I}{D(i)} \ar[r]^-{F^e_{\lambda}} \ar[dr]_-{\lim_{i \in I}{F^e_{D(i)/S}}} & \bigl(\lim_{i \in I}{D(i)} \bigr)^{(q)}\\
& \lim_{i \in I}{D(i)^{(q)}} \ar[u]^-{\cong}_-{\textrm{can}}
}
\]
Then, if every $F^e_{D(i)/S}$ is an isomorphism then so is $F^e_{\lambda}$. It remains to explain why $\lim_{i \in I}{D(i)}$ is flat over $S$. We presume this is well-known yet we prove it here for the sake of completeness and by the lack of a suitable reference. First of all, notice that limits are stable under base change---``limits commute with limits''---then we may assume the base $S$ is affine, say $S = \Spec R$. Then, flatness follows from the exactness of colimits over filtered categories in $R$\textsf{-mod} \cite[\href{https://stacks.math.columbia.edu/tag/00DB}{Tag 00DB} and \href{https://stacks.math.columbia.edu/tag/04B0}{Tag 04B0}]{stacks-project}. Indeed, this implies that colimits over filtered categories of flat $R$-modules are flat.
\end{proof}

\begin{example}[Henzelizations]
    The (strict-)henselization $R \to R^{\mathrm{(s)h}}$ of a local ring $(R,\fram,\kay)$ is ind-\'etale and thus pristine. 
\end{example}

\usetikzlibrary{decorations.markings}
\tikzset{degil/.style={
            decoration={markings,
            mark= at position 0.5 with {
                  \node[transform shape] (tempnode) {$\backslash$};
                  }
              },
              postaction={decorate}
}
}
Ind-étale homomorphisms belong to the larger class of \emph{weakly \'etale} homomorphisms. Recall that a morphism $f\colon X \to S$ is weakly \'etale if both $f$ and the diagonal $X \to X \times_S X$ are flat. Equivalently, $f$ is weakly \'etale if and only if the local homomorphisms $f^{\#}_x \: \sO_{S,f(x)} \to \sO_{X,x}$ are isomorphisms on strict henselizations; see \cite[\href{https://stacks.math.columbia.edu/tag/094Z}{Tag 094Z}]{stacks-project}. Additionally, a morphism $f$ is called \emph{$\bL$-trivial} if its cotangent complex $\bL_f$ has trivial homology.

Our next goal is to establish that for a morphism $f\colon X \to S$ of schemes, we have the following strict implications:
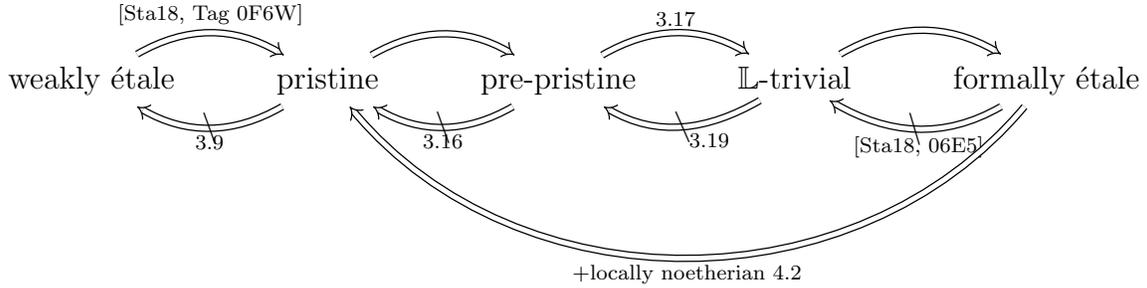
\[\begin{tikzcd}[arrows=Rightarrow]
\text{weakly \'etale} \arrow[out=30,in=150, "\text{\cite[\href{https://stacks.math.columbia.edu/tag/0F6W}{Tag 0F6W}]{stacks-project}}"]{r}{}& \text{pristine} \arrow[out=210,in=330, degil, "\ref{rem.notweaklyetalebutpristine}"]{l}{}  \arrow[out=30,in=150]{r}{}& \text{pre-pristine} \arrow[out=30,in=150, "\ref{pro.Pre-pristineHavetrivialCotangentComplex}"]{r}{} \arrow[out=210,in=330, degil, "\ref{ex.PrePristineNotPristine}"]{l}{}& \mathbb{L}\textnormal{-trivial} \arrow[out=30,in=150]{r}{} \arrow[out=210,in=330, degil, "\ref{ex.BhattGabberExample}"]{l}{}& \text{formally \'etale} \arrow[out=210,in=330, degil, pos=0.48, "\text{\cite[\href{https://stacks.math.columbia.edu/tag/06E5}{06E5}]{stacks-project}}"]{l}{} \arrow[out=220,in=310, bend left=50, "+ \text{locally noetherian \ref{fetaleprepristinefornoetherian}}"]{lll}{}
\end{tikzcd}\]

\begin{remark}[Pristinity does not imply weak étaleness]
\label{rem.notweaklyetalebutpristine}
We explain why pristinity does not imply being weakly \'etale. Let $(R,\fram,\kay)$ be an $F$-finite strictly henselian noetherian local ring. By \autoref{pro.completionpristineFfinite}, we know that its completion $R \to \hat{R}$ is pristine but for it to be weakly \'etale it must be an isomorphism (\cite[\href{https://stacks.math.columbia.edu/tag/094Z}{Tag 094Z}]{stacks-project}). For a concrete example, take $R$ to be the (strict-)henselization of $\bF_p[x]$ at $(x)$, which is the subring of $\bF_p\llbracket x \rrbracket$ of algebraic power series and so not complete.
\end{remark}

\begin{remark} \label{rem.LTrivialImpliesFormallyEtale} The fact that $\bL$-trivial maps are formally étale can be seen as follows. By \cite[\href{https://stacks.math.columbia.edu/tag/08UR}{Lemma 08UR}]{stacks-project}, $f$ is formally unramified. By \cite[\href{https://stacks.math.columbia.edu/tag/02HH}{Lemma 02HH}]{stacks-project}, it suffices to show that $f$ is formally smooth, for which we may assume that it is affine (\cite[\href{https://stacks.math.columbia.edu/tag/0D0F}{Lemma 0D0F}]{stacks-project}). The assertion now follows from \cite[\href{https://stacks.math.columbia.edu/tag/031J}{Proposition 031J}]{stacks-project}, \cite[\href{https://stacks.math.columbia.edu/tag/08RB}{Lemma 08RB}]{stacks-project}, and \cite[\href{https://stacks.math.columbia.edu/tag/08T3}{Lemma 08T3}]{stacks-project}.
Furthermore, if we assume that $f$ is finitely presented, then all these notions are equivalent to $f$ being 
\'etale (use \cite[\href{https://stacks.math.columbia.edu/tag/02HM}{Lemma 02HM}]{stacks-project} and \cite[\href{https://stacks.math.columbia.edu/tag/094X}{Lemma 094X}]{stacks-project}).
\end{remark}

\begin{remark}[$p$-infinitesimal deformations]
\label{PrepristineFormallyunramified}
    The above shows a path to see why pre-pristine morphisms are formally étale. However, this can be checked directly and completely elementarily (without resorting to using the cotangent complex). For example, since $\Omega_f = \Omega_{F_f}$, we see that as soon as $F_f$ is a closed immersion, $f$ is formally unramified. Likewise, one verifies the infinitesimal lifting property for $f$ if $F_f$ is a closed immersion. In fact, let $\alpha \: R \to A$ be an $R$-algebra. Consider a commutative diagram of $R$-algebras
\[
\xymatrix{
R \ar[d]_-{\alpha} \ar[r]^-{\beta} & B \ar[d]^-{\pi} \\
A \ar[r]^-{\theta } & B/\frab
}
\]
such that $\frab^{[p]}=0$ (e.g. $\frab^2=0$). Let $\vartheta_0 \: A \to B$ be a homomorphism of $R$-algebras lifting $\theta: A \to B/\frab$. Then, $F_*\frab$ is canonically a $B/\frab$-module and so an $A$-module by restriction of scalars along $\theta$, which makes it into an $A_{F_*R}$-module. Moreover:
\begin{claim}
    We have the following commutative diagram
   \[
   \xymatrix@C=6em{
 \{\vartheta \in \Hom_{R\textnormal{-alg}}(A,B) \mid \theta = \pi \circ \vartheta \} \ar[r]^-{\xi\:\vartheta \mapsto F_*(\vartheta-\vartheta_0)} \ar[rd]_-{\subset} & \ker\big( \Hom_{A_{F_*R}}(F_*A,F_*\frab) \xrightarrow{\phi \mapsto  \phi(F_*1)} F_*\frab \big) \ar[d]^-{\zeta\:\phi \mapsto \theta_0 + \varphi} \\
   & \{\vartheta \in \Hom_{R}(A,B) \mid \theta = \pi \circ \vartheta, \vartheta(1)=1\}
   }
   \]
    where $\varphi \: A \to B$ is defined by the identity $F_* \varphi(a) = \phi (F_* a)$ for all $a \in A$. In particular, $\xi$ is injective.
\end{claim}
\begin{proof}[Proof of claim]
    notice that $\phi \coloneqq \vartheta - \vartheta_0$ indeed maps $A$ into $\frab \coloneqq \ker \pi$ and $\phi(F_*1)=0$. Moreover, this difference is $A^p \otimes_{R^p} R$-linear. The fact that it is $R$-linear is obvious, and for all $a,a'\in A$ we have
    \begin{align*}
     \phi(a^p a') &= \vartheta(a)^p \vartheta(a') - \vartheta_0(a)^p \vartheta_0(a')\\
     &=\vartheta(a)^p \vartheta(a') - \vartheta(a)^p \vartheta_0(a') + \vartheta(a)^p \vartheta_0(a') - \vartheta_0(a)^p \vartheta_0(a') \\
     &= \vartheta(a)^p \phi(a') + \phi(a)^p \vartheta_0(a') \\
     & = \vartheta(a)^p \phi(a')
    \end{align*}
    as $\phi(a)^p \in \frab^{[p]} = 0$. So $\xi$ is a well-defined function. 
\end{proof}
In particular, if $A_{F_*R} \to F_*A$ is surjective then the codomain of $\xi$ is zero and so its domain is a singleton set $\{\vartheta_{0}\}$, which is to say that $\alpha$ is b-nil formally unramified \cite{morrownotes}.
\end{remark}

\begin{example}[Formally \'etale does not imply pristine] \label{ex.formallyEtaleNotPristine}
Here is another interesting example in addition to \cite[\href{https://stacks.math.columbia.edu/tag/06E5}{06E5}]{stacks-project} mentioned in the diagram above. Let $p \neq 2$ and $R$ be the ring
\[
R \coloneqq \bF_p \bigl[t^{a/2^b} \bigm| a,b \in \bZ_{> 0} \bigr].
\]
Let $\theta \: R \to A$ be the quotient algebra given by the ideal $\mathfrak{a} \coloneqq \big(t^{a/2^b} \bigm| a,b \in \bZ_{>0}\bigr) \subset R$, in particular $A \cong \bF_p$. We claim that $\theta$ is formally \'etale but not pristine.
\begin{claim}
$\theta \: R \to A$ is formally \'etale.
\end{claim}
\begin{proof}
This follows from the equality $\mathfrak{a}^2 = \mathfrak{a}$. Indeed, let $\psi \: R \to B$ be some $R$-algebra containing an ideal $I \subset B$ such that $I^2=0$, we must prove that the function
\[
\Hom_{R\textnormal{-alg}}(A,B) \to \Hom_{R\textnormal{-alg}}(A,B/I), \quad \theta \mapsto \pi \circ \theta 
\]
is bijective, where $\pi \: B \to B/I$ is the canonical quotient homomorphism. Observe that this is a function between sets of cardinality $0$ or $1$. The injectivity follows since $\theta$ is clearly formally unramified as $\Omega_\theta = 0$. To see it is surjective, suppose there is an $R$-algebra homomorphism $A \to B/I$. This simply means $(\pi \circ \psi)(\mathfrak{a}) = 0$, \ie $\psi(\mathfrak{a}) \subset \psi(\mathfrak{a}) \cdot B \subset I$. However, by taking squares, this inclusion yields
\[
\psi(\mathfrak{a}) = \psi \big(\mathfrak{a}^2\big) \subset \big(\psi(\mathfrak{a}) \cdot B \big)^2 \subset I^2 = 0.
\]
In other words, there is a homomorphism of $R$-algebras $A \to B$.
\end{proof}

\begin{claim} \label{cla.NotFlatNorPrePristine}
$\theta \: R \to A$ is not flat nor pre-pristine.
\end{claim}
\begin{proof}
To see why $\theta$ is not flat, note that $\mathfrak{a}$ is not a pure ideal; \cite[\href{https://stacks.math.columbia.edu/tag/04PQ}{Tag 04PQ}]{stacks-project}. Indeed, observe that, for the element $t \in \mathfrak{a}$, it is impossible to find $y \in \mathfrak{a}$ such that $t=yt$.

Similarly, we see directly that the relative Frobenius of $\theta$ is not an isomorphism. In fact, for a closed immersion like $\theta \: R \to A = R/\mathfrak{a}$, the relative Frobenius is given by
\[
F^e_{\theta} \: F^e_* \big(R/\mathfrak{a}^{[q]}\big) \to F^e_* (R/\mathfrak{a}) = F^e_* \big( R/\mathfrak{a}^{[q]} \xrightarrow{\mathrm{can}} R/\mathfrak{a} \big).
\]
However, in our case, the canonical homomorphism $R/\mathfrak{a}^{[q]} \to R/\mathfrak{a}$ is not an isomorphism. Indeed, $R/\mathfrak{a} = A \cong \bF_p$, whereas $R/\mathfrak{a}^{[q]}$ is an infinite dimensional $A \cong \bF_p$-module (for instance, the elements $t,t^{1/2},t^{1/4},t^{1/8},\ldots$ form a linearly independent set).
\end{proof}
\end{example}

\begin{example} [Pre-pristinity does not imply pristinity] \label{ex.PrePristineNotPristine}
In this example, we prove that pre-pristine morphisms are not necessarily pristine. Indeed, it suffices to take $R \to A$ as in \autoref{ex.formallyEtaleNotPristine} with $p=2$. More generally, for an arbitrary prime $p$ let 
\[
R \coloneqq \bF_p \big[\bZ[1/p] \cap \bQ_{> 0}\big] = \bF_p \bigl[t^{a/p^b} \bigm| a,b \in \bZ_{> 0} \bigr]
\]
and let $\theta \: R \to A$ be the closed immersion cut out by the ideal $\mathfrak{a} \coloneqq \big(t^{a/p^b} \bigm| a,b \in \bZ_{>0}\bigr) \subset R$. One readily verifies that $\mathfrak{a}^{[p]} = \mathfrak{a}$, and therefore $R \to A$ is pre-pristine. However, $\mathfrak{a}$ is not a pure ideal (by the same reason as in the proof of \autoref{cla.NotFlatNorPrePristine}), thus $R \to A$ is not flat.
\end{example}

\begin{proposition} \label{pro.Pre-pristineHavetrivialCotangentComplex}
Let $f\: X \to S$ be a morphism of schemes. Then, we have a canonical quasi-isomorphism
\[
\bL_{F^e_f} \cong \bL_f \oplus F^{e,*}_X \bL_f[1]
\]
whose limit as $e \rightarrow \infty$ yields 
\[
\bL_{F_f^{\infty}} \cong F^{\infty,*}_X \bL_f[1].
\] 
In particular, pre-pristine morphisms are $\bL$-trivial and so formally \'etale.
\end{proposition}
\begin{proof}
By \cite[\href{https://stacks.math.columbia.edu/tag/08QR}{Tag 08QR}]{stacks-project}, to the triangle of schemes
\[
\xymatrix{
X \ar[rr]^-{F^e_{f}} \ar[rd]_-{f} & & X^{(q)} \ar[ld]^-{f^{(q)}}  \\
& S &
}
\]
we may associate its fundamental distinguished triangle in the derived category $\mathsf{D}(X)$
\[
F^{e,*}_{f} \bL_{f^{(q)}} \xrightarrow{ \mathrm{d}F^e_{f} = 0} \bL_{f} \to \bL_{F_f^e} \to F^{e,*}_{f} \bL_{f^{(q)}}[1] 
\]
Since the map most to the left is zero, this implies
\[
\bL_{F^e_f} \cong \bL_{f} \oplus F^{e,*}_{f} \bL_{f^{(q)}}[1] \cong \bL_f \oplus F^{e,*}_X \bL_f[1]
\]
as cotangent complexes commute with base change.

In taking the limit as $e \rightarrow \infty$, the direct summand $\bL_f$ is simply replaced by $\bL_{f_{\mathrm{perf}}}$, which is zero as $f_{\mathrm{perf}}$ is a morphism between perfect schemes.
\end{proof}

\begin{remark}
    It is not true that a morphism $f$ such that $F^{\infty}_f$ is an isomorphism is $\bL$-trivial (nor formally \'etale) unless $F_S$ (equivalently $F_X$) is flat. Use the same example from \autoref{rem.RelativePerfectionBeingIsoDoesNotImpolyFrobIsISo}. We hope this convinces the reader that $F^{\infty}_f$ being an isomorphism is not a good notion of ``pristinity.''
\end{remark}

The converse of the last implication in \autoref{pro.Pre-pristineHavetrivialCotangentComplex} does not hold as the following example due to B.~Bhatt and O.~Gabber shows.

\begin{example}[Bhatt--Gabber example of a non-perfect algebra with trivial cotangent complex] \label{ex.BhattGabberExample}
The following example is extracted from \cite{BhattGabberExampleNonReducedTrivialCotangentBundle}. In that note, B.~Bhatt presented an example of an imperfect ring defined over a perfect field with trivial cotangent complex. In particular, formally \'etale (nor triviality of the cotangent complex) does not imply pristine even in the flat case. The algebra is constructed as follows. Let $\kay$ be a perfect field of characteristic $p$. For every $i \in \mathbb{N}$, let $B_i = \kay [x_{i,1},\ldots,x_{i,2^i}]$ and consider the ideal $\fram_i \coloneqq (x_{i,1},\ldots,x_{i,2^i}) \subset B_i$. Define:
\[
A_i \coloneqq B_{i,\textnormal{perf}} /\fram_i \cdot B_{i,\textnormal{perf}}
\]
where 
\[
B_{i,\textnormal{perf}} \coloneqq \colim_{F} B_i = \colim (B_i \xrightarrow{F} B_i \xrightarrow{F} B_i \to \cdots) \cong \kay\Big[x_{i,1}^{1/p^{\infty}},\ldots,x_{i,2^i}^{1/p^{\infty}}\Big]. 
\]
Thus, more succinctly, we have that:
\[
A_i = \kay\Big[x_{i,1}^{1/p^{\infty}},\ldots,x_{i,2^i}^{1/p^{\infty}}\Big]\Bigm/(x_{i,1},\ldots,x_{i,2^i}) \eqqcolon \kay\Big[\varepsilon_{i,1}^{1/p^{\infty}},\ldots,\varepsilon_{i,2^i}^{1/p^{\infty}}\Big].
\]
Next, consider the homomorphism
\[
\theta_i \: A_i \to A_{i+1}, \quad \varepsilon_{i,j}^{1/p^e} \mapsto (\varepsilon_{i+1,2j-1} \cdot \varepsilon_{i+1,2j})^{1/p^e}.
\]
Then, the $\kay$-algebra $A \coloneqq \colim_i A_i$ is a non-reduced algebra with trivial cotangent complex; see \cite{BhattGabberExampleNonReducedTrivialCotangentBundle} for details. In particular, $F_A$ is not injective, so that $F_{A/\kay}$ cannot be injective either (as $F_{\kay}$ is an isomorphism). Hence, this is a formally \'etale but not pre-pristine algebra.

Finally, we would like to remark that it can be seen in a more elementary fashion that $A/\kay$ is formally \'etale. Clearly, $F_{A_i}$ is surjective and thus also $F_A$. Since $F_\kay$ is an isomorphism it follows that $F_{A/\kay}$ is a closed immersion. Hence, $A/\kay$ is formally unramified by \autoref{PrepristineFormallyunramified}. To see it is formally smooth, we consider a $\kay$-algebra $B$ containing an ideal $I\subset B$ whose square is zero and denote by $\pi \: B \to B/I$ the corresponding quotient algebra. Suppose there is a homomorphism of $\kay$-algebras $\vartheta \: A \to B/I$. We wish to lift it to a $\kay$-homomorphism $\theta \: A \to B$. To this end, notice $\vartheta \: A \to B/I$ is nothing but a collection of $\kay$-homomorphisms $\vartheta_i \: A_i \to B/I$ satisfying commutative diagrams
\[
\xymatrix{
A_i \ar[rr]^-{\theta_i} \ar[rd]_-{\vartheta_i} && A_{i+1} \ar[ld]^-{\vartheta_{i+1}} \\
& B/I &
}
\]
However, this simply corresponds to choosing elements $b(i,j,e) \in B$ satisfying relations
\[
b(i,j,1)^p \equiv 0, b(i,j,e+1)^p \equiv b(i,j,e) \bmod I
\]
and moreover
\[
b(i,j,e) \equiv b(i+1,2j-1,e)b(i+1,2j,e) \bmod I.
\]
Indeed, the former two relations correspond to the specification of a $\kay$-homomorphism $\vartheta_i \: A_i \to B/I$, and the latter relation corresponds to the compatibility of these with the transition homomorphisms $\theta_i$. However, the crucial point is that the latter relations force the former to hold $\bmod I^2$ and so to hold in $B$. More precisely, we have
\[
b(i,j,1)^p = 0, \quad b(i,j,e+1)^p = b(i,j,e). 
\]
This simply means that each $\vartheta_i \: A_i \to B/I$ lift to a $\kay$-homomorphism $\theta_i \: A_i \to B$, i.e. $\vartheta_i =\pi \circ \theta_i$. Nonetheless, we do not know \emph{a priori} whether these lifted homomorphisms are compatible with the transition homomorphism $\theta_i \: A_i \to A_{i+1}$, that is, we do not know if $\theta_i = \theta_{i+1} \circ \theta_i$. However, we know that $\theta_i$ and $\theta_{i+1} \circ \theta_i$ coincide after post-composing them with $\pi$. Since the algebras $A_i/\kay$ are themselves formally unramified (since their relative Frobenii are surjective), we see that $\theta_i$ and $\theta_{i+1} \circ \theta_i$ must agree, defining in this way a $\kay$-homomorphism $\theta : A \to B$ so that $\vartheta = \pi \circ \theta$.
\end{example}

\begin{remark}
In \autoref{ex.formallyEtaleNotPristine} we found a first example of a formally \'etale morphism that is not pristine. The fundamental failure for this was that the morphism was not flat. Na\"ively, one might hope that assuming flatness formal \'etaleness may imply pristinity. However, \autoref{ex.BhattGabberExample} shows that this is not possible. Similarly, this example of Bhatt--Gabber shows that a flat formally smooth morphism does not necessarily have flat relative Frobenius (otherwise it would be injective on sections).
\end{remark}

\section{Noetherian Pristinity} \label{sec.NoetherianPristinity}
As we have just seen, the class of (pre-)pristine morphisms fits nicely in between the class of weakly \'etale morphisms and that of formally \'etale morphisms. In this subsection, we specialize to the case of morphisms $f\: X \to S$ between locally noetherian schemes. 

We have already seen how if $F_f$ is a closed immersion (resp. open immersion), then $f$ is formally unramified (resp. formally \'etale). We proceed to prove the converse in the locally noetherian case.

\begin{theorem}[Kunz-type theorem for formal unramification]
\label{NoetherianUnramifiedRelFClosedImmersion}
Let $f \: X \to S$ be a morphism of schemes such that $F_f\: X \to X^{(p)}$ is a morphism of locally noetherian schemes. If $f$ is formally unramified, then $F_f$ is a closed immersion.
\end{theorem}
\begin{proof}
Since $G_f \circ F_f = F_{X}$, $F_f \circ G_f = F_{X^{(p)}}$, and $\Omega_{F_f}=\Omega_f = 0$ (\cite[\href{https://stacks.math.columbia.edu/tag/02H9}{Lemma 02H9}]{stacks-project}), we may apply \cite[Theorem 11.2]{polstramafsingularities} to conclude that $F_f$ is a closed immersion (\cf \cite{TycDifferentialBasisAndp-basis} and \cite[Remark 11.1]{polstramafsingularities}).
\end{proof}

\begin{corollary}[Kunz-type theorem for formal étaleness]
\label{fetaleprepristinefornoetherian}
Let $f\colon X \to S$ be a morphism of locally noetherian schemes. Then $f$ is pristine if (and only if) $f$ is formally \'etale.    
\end{corollary}
\begin{proof}
Suppose that $f$ is formally \'etale, making it regular as pointed out in \autoref{rem.formallysmoothcharacterization}. Hence, $X^{(p)}$ is locally noetherian by \autoref{theo.raduandrecore}. We may then use \autoref{NoetherianUnramifiedRelFClosedImmersion} to conclude that $F_f$ is a closed immersion. Moreover, $F_f$ is a flat closed immersion (using \autoref{theo.andreraduthm}) of locally noetherian schemes, and so we are done by \autoref{rem.WeakiningsOnF_fForPrisitinity}. 
\end{proof}

\begin{remark}
    Daniel Fink \cite{FINKRelativeInverseLimitPerfection} has recently obtained a proof of \autoref{fetaleprepristinefornoetherian} in the $F$-finite case by using derived-theoretic methods. Similarly, the same result has been obtained independently at the same time by Datta--Olander \cite{DataOlanderFIso}.
\end{remark}

We conclude our remarks on noetherian pristinity with the following.

\begin{proposition}
\label{theo.pristinecharacterizationnoetherian}
Let $f\: X \to S$ be a morphism between locally noetherian schemes.
\begin{enumerate}

\item{$f$ is pristine if and only if it is $F$-finite, flat, and all fibers $f_s \:X_s \to s$ are pristine.} 
\item{If $f$ is formally \'etale (pristine), then $f$ has relative dimension zero if and only if the residual field extensions $\kappa(x)/\kappa(f(x))$ are separable (i.e. formally smooth) for all $x \in X$. In that case, the fibers of $f$ are locally spectra of finite products of fields and so these are pristine.}
\item If $f$ is formally \'etale and fiberwise $F$-finite, then $f$ has pristine fibers of relative dimension zero.
\item Let $f'\: X'\to S'$ be a base change of $f$ that is a morphism of locally noetherian schemes. If $f$ is an $F$-finite regular morphism, then so is $f'$. 
\end{enumerate}
\end{proposition}
\begin{proof}
The ``only if'' direction of (a) is immediate from \autoref{prepristineproperties} (a). Conversely, since $F_f$ is a universal homeomorphism, it suffices to show that $F_f$ is an open immersion. To this end, we show that $F_f$ is flat and a closed immersion. To show that $F_f$ is flat, note that $f$ is flat with geometrically regular fibers. So the assertion follows from Radu--Andr\'e's theorem. To show that $F_f$ is a closed immersion, it suffices to show $\Omega_f = 0$ (see \autoref{NoetherianUnramifiedRelFClosedImmersion}). Since this sheaf is coherent, this can be checked on fibers where it is immediate from the hypothesis.

For (b), we may assume that the base of $f$ is a field and further that the source is affine. Thus, let $A/\kay$ be a (noetherian) formally \'etale algebra, which is geometrically regular. Given a maximal ideal $\fram \subset A$, consider the diagram
\[
\xymatrix{
\Spec \kappa(\fram)\ar[rr] \ar[dr]_-{\text{f.un}}  &  &  \Spec A  \ar[dl]^-{\text{f.\'et}}\\
& \Spec \kay  &
}
\]
where we readily see that $\kappa(\fram)/\kay$ is formally unramified as so is $A/\kay$ (use any of the two exact sequences regarding K\"ahler differentials). Then, this makes clear that $\kappa(\fram)/A$ is formally \'etale (i.e. $\fram/\fram^2=0$) if and only if $\kappa(\fram)/\kay$ is formally smooth; see \cite[17.1.5]{EGAIV}. However, since $A$ is noetherian, $\kappa(\fram)/A$ is formally \'etale if and only if $A_{\fram}$ is a regular local ring of dimension zero (i.e. $A_{\fram} \to \kappa(\fram)$ is an isomorphism).

For (c), we may immediately reduce to $f_s\: \Spec A \to \Spec \kay$ formally \'etale with $A$ an $F$-finite noetherian ring. This implies that $F_{f_s}$ is finite and thus $f_s$ is pristine by \autoref{fetaleprepristinefornoetherian}. We may localize at a maximal ideal $\fram$ and, using \autoref{pro.completionpristineFfinite}, pass to the completion $\hat{A}$ of $A$ at $\fram$. By the Cohen structure theorem, $\hat{A} \cong \el \llbracket x_1,\ldots,x_d \rrbracket$, where $d=\dim \hat{A}$. Assume for the sake of contradiction that $d>0$. By surjectivity of $F_{\hat{A}/\kay}$, we can write $x_d$ as a finite sum $x_d = \sum_i g_i^p \cdot c_i$ where $g_i \in \el \llbracket x_1,\ldots,x_d \rrbracket$ and $c_i \in \kay$. Since the coefficient of $x_d$ on the right-hand side is zero, this is a contradiction.

For (d), note that $f'$ is $F$-finite by \cite[Lemma 2 (5)]{HashimotoFfinitenessAndItsdescent}. Hence, the regularity of $f'$ is equivalent to its formal smoothness, the latter base changes by \cite[\href{https://stacks.math.columbia.edu/tag/02H2}{Lemma 02H2}]{stacks-project}.
\end{proof}

\section{Pristine Pullback of Cartier Modules}
\label{sect.pristinepullbacks}
In this section, we construct a pullback functor $\sC \text{-}\mathrm{mod} \to f^\ast \sC \text{-}\mathrm{mod}$ for a pristine morphism $f\: X \to S$ between $F$-finite locally noetherian schemes and a Cartier algebra $\sC$ on $S$. We start by briefly recapping the notions of Cartier module and Cartier algebra.

\begin{definition}[Cartier algebras and modules]
Let $X$ be a locally noetherian scheme.
\begin{enumerate}[(a)]
\item  A \emph{Cartier algebra} $\sC_X$ is a sheaf of $\mathbb{N}$-graded rings on $X$ with $\sC_{X,0} = \sO_X$ that is a quasi-coherent $\sO_X$-bimodule with $r \cdot \varphi = \varphi \cdot r^{q}$ for local sections $\varphi \in \sC_{X,e}$.
\item A \emph{$\sC_X$-module} $\sF$ is a left $\sC_X$-module for which the underlying $\sO_X$-module structure is coherent. We denote the category of $\sC_X$-modules by $\sC_X \text{-}\mathrm{mod}$ and may refer to them simply as \emph{Cartier modules}.
\end{enumerate}
\end{definition}

Note that (\cite[Lemma 5.2]{BlickleStablerFunctorialTestModules}) a $\sC_X$-module $\sF$ is the same as a coherent $\sO_X$-module $\sG$ together with a graded homomorphism of rings 
\[
\Xi\: \sC_X \to \bigoplus_{e\geq 0} \ssHom(F_\ast^e \sG, \sG).
\]

\begin{lemma}
\label{le.pristinenimodulepullback}
Let $M$ be an $R$-bimodule such that $r\cdot m = m\cdot r^q$. If $\alpha\: R \to A$ is a pristine homomorphism of rings, then the left $A$-module $A \otimes_R M \eqqcolon M_A$ is endowed with a right $A$-module structure via the $A$-module isomorphism $F_{\alpha}\: A \otimes_R F^e_* R \to F^e_* A$. Moreover, $M_A$ is an $A$-bimodule such that $a\cdot m = m \cdot a^q$.
\end{lemma}
\begin{proof}
  We can think of $M$ as an $R$-module on the left and an $F^e_* R$-module on the right, given by $m \cdot F^e_* r \coloneqq m \cdot r$. For $r, s \in R$ and $m \in M$, it follows that
  \[
    (r \cdot m) \cdot F^e_* s = (m \cdot r^q) F^e_* s = m \cdot (F^e_* r^q s) =  m \cdot (r F^e_* s).
    \]
  In particular, $M$ is a right $R \otimes_R F^e_* R$-module. It follows that $A \otimes_R M$ is a right $A \otimes_R F^e_* R$-module. By definition, for any $a, b \in A$, $m \in M$, and $r \in R$, we have $(a \otimes m)(b \otimes F^e_* r) = ab \otimes m F^e_* r$. As the relative Frobenius is an isomorphism, this gives a right $F^e_* A$-module action on $M_A$, namely $(a \otimes m) \cdot F^e_* b \coloneqq (a \otimes m) (F_{\alpha}\invrs(b))$ for all $b\in A$.  Note that the relative Frobenius is given by $F_{A/R}(a \otimes F^e_* r) = F^e_* a^q r$ for all $a\in A$ and $r \in R$. It follows that $(a \otimes m) \cdot F^e_* b^q = (a \otimes m) \cdot F_{A/R}(b \otimes 1) = ab \otimes m = b \cdot (a \otimes m)$; as desired. 
\end{proof}

This enables us to define a pristine pullback functor for Cartier modules as follows.

\begin{proposition}
\label{pro.daggerpullback}
Let $f\: X \to S$ be a pristine morphism of locally noetherian $F$-finite schemes and $\sC$ a Cartier algebra on $S$. The pull back $f^\ast \sC$ as a left $\sO_S$-module can naturally be endowed with a Cartier algebra structure via the formula
\begin{align}\label{eq.pristineringstructure} (a \otimes \kappa) \cdot (b \otimes \delta) = a \otimes \kappa F_\ast^e b \delta = \sum_i a b_i \otimes \kappa F_\ast^e r_i \delta
\end{align} for local sections $a,b =\sum_i r_i b_i^q \in \sO_X, r_i \in \sO_S$, $\kappa \in \sC_{e}$ and $\delta \in \sC$.
Moreover, $f^\ast$ induces a functor $\sC$-$\mathrm{mod} \to f^* \sC$-$\mathrm{mod}$.
\end{proposition}
\begin{proof}
Arguing just as in \cite[Proposition 5.3]{BlickleStablerFunctorialTestModules}, we may reduce to the situation of a pristine morphism $f\: \Spec A \to \Spec R$ with the corresponding ring map $\alpha\: R \to A$. It is immediately clear from \autoref{le.pristinenimodulepullback} that $f^\ast \sC$ carries an $A$-bimodule structure with $a \cdot \kappa = \kappa \cdot a^q$ for $\kappa \in \sC_{e}$. A small computation shows that \autoref{eq.pristineringstructure} induces a ring structure.

Recall that the endowment of an $\sO_S$-module $\sF$ with a $\sC$-module structure is equivalent to specifying a homomorphism of graded rings 
    \begin{equation} \label{eq.Cartiermodulestructre}\Phi\: \sC \to \bigoplus_{e \geq 0} \ssHom_S \bigl(F^e_* \sF, \sF\bigr). \end{equation}
We will construct a canonical isomorphism 
\[
f^\ast \ssHom_S \bigl(F^e_* \sF, \sF\bigr) \cong \ssHom_X \bigl(F^e_* f^\ast \sF, f^\ast \sF\bigr).
\] 
With this isomorphism in place, applying $f^\ast$ to \autoref{eq.Cartiermodulestructre} and using this canonical isomorphism then defines the $f^\ast \sC$-module structure on $f^\ast \sF$. 

To obtain the desired isomorphism, consider the cartesian diagram \autoref{eqn.ReelativeFrobeniusDiagram}, from which we obtain the natural isomorphism
\[
f^\ast F^e_{S \ast} \cong G^e_{f\ast} f^{(q)\ast}.
\] 
From our assumption that $F^{e}_f$ is an isomorphism, we obtain a natural isomorphism 
\[
F^e_{f\ast} F_f^{e\ast} \cong \id_{X^{(q)}}.
\] 
Thus, we have natural isomorphisms
\[F^e_{X\ast} f^\ast \cong G^{e}_{f\ast} F^e_{f \ast} f^\ast \cong G^{e}_{f\ast} F^e_{f\ast} F_f^{e\ast} f^{(q)\ast} \cong G^{e}_{f\ast} f^{(q)\ast} \cong f^\ast F^e_{S\ast}. \]
Using the fact that $f$ is flat and $F_\ast^e \sF$ is coherent (since $F_\ast^e$ is finite), the canonical map
\[
f^\ast \ssHom_S(F_{S\ast}^e \sF, \sF) \to \ssHom_X(f^\ast F_{S\ast}^e \sF, f^\ast \sF).
\]
is an isomorphism. By the natural isomorphisms just discussed, we finally obtain the desired isomorphism
\[ 
\ssHom_X(f^\ast F_{S\ast}^e \sF, f^\ast \sF) \cong  \ssHom_X( F_{X\ast}^e f^\ast \sF, f^\ast \sF).
\]

We still need to argue that $f^\ast \Phi$ in \autoref{eq.Cartiermodulestructre} does indeed define a ring homomorphism. This can be checked locally on an open affine cover, so let $f\: \Spec A \to \Spec R$ be pristine and $M$ a coherent $\sC$-module. For homogeneous elements $\kappa, \delta$ of degree $e, e'$ in $\sC$ and $a, b = \sum_i r_i b_i \in A$, an explicit computation reveals that
\[
f^\ast \Phi\Biggl(\sum_i a b_i \otimes \kappa \cdot F_\ast^e r_i \cdot \delta\Biggr) = f^\ast \Phi(a \otimes \kappa) \circ F_\ast^e f^\ast \Phi(b \otimes \delta);  
\]
as required.
\end{proof}

We note that our construction of $f^* \sC$ is different from that in \cite[\S 5]{BlickleStablerFunctorialTestModules}, where the pullback is constructed as a pullback of \emph{right} $\sO_S$-modules. Denoting our pristine pullback momentarily by $f^{\dagger}$, we next show that there is a natural isomorphism $f^\dagger \to f^\ast$.

Given a Cartier algebra $\sC$, we will write $\prescript{\mathbf{l}}{}\sC$ for the underlying left $\sO_X$-module and $\sC^\mathbf{r}$ for the underlying right $\sO_X$-module.  For a commutative ring $R$, we may view a right $R$-module $M$ as a left $R$-module $M^\mathrm{op}$ according to the rule $r \cdot m \coloneqq m \cdot r$. We will similarly express the change from a left $R$-module $N$ to a right $R$-module by $N^{\mathrm{op}}$. That a ring $\sC$ with an $\sO_X$-bimodule structure is a Cartier algebra may then be expressed by $(F_\ast^e \sC^{\mathbf{r}})^{\mathrm{op}} =\prescript{\mathbf{l}}{}{\sC}$.

\begin{proposition}
\label{pro.Cartierstructurecomparison}
With notation as in \autoref{pro.daggerpullback}, there is a natural isomorphism $f^{\dagger} \sC \to f^* \sC$ of Cartier algebras given locally by $a \otimes \varphi \mapsto \varphi \otimes a^q$ for all $a\in \sO_X(U)$ and $\varphi \in \sC(U)$ on an open $U \subset X$.
\end{proposition}
\begin{proof}
 The Cartier algebra $f^* \sC$ is defined as a right $\sO_X$-module by $f^* \sC^{\mathbf{r}}$ and its left $\sO_X$-module structure is $(F^e_* f^* \sC^{\mathbf{r}})^{\mathrm{op}}$. On the other hand, $f^{\dagger} \sC$ is defined so that
 \[
 \big(\prescript{\mathbf{l}}{}{f}^{\dagger} \sC\big)^{\mathrm{op}} = \big(f^* {^{\mathbf{l}}}\sC\big)^{\mathrm{op}} \cong f^* F^e_* \sC^{\mathbf{r}} .
 \]
By pristinity, we have a canonical isomorphism of $\sO_X$-modules $f^* F^e_* \sC^{\mathbf{r}} \to F^e_* f^* \sC^{\mathbf{r}}$ given on local sections by $\varphi \otimes a \mapsto \varphi \otimes a^q$, just as in the proof of \autoref{pro.daggerpullback}. In other words, we have the following canonical isomorphism of $\sO_X$-modules
\begin{align*}
 \big(\prescript{\mathbf{l}}{}f^{\dagger} \sC\big)^{\mathrm{op}} &\xlongrightarrow{\cong} F^e_* f^* \sC^{\mathbf{r}}\\
 a \otimes \varphi &\longmapsto \varphi \otimes a^q.
 \end{align*}
 By definition, the right $\sO_X$-module structure of $f^{\dagger} \sC$ and the left $\sO_X$-module structure of $f^*\sC$ are defined through this canonical isomorphism. In other words, the above canonical isomorphism becomes an isomorphism $f^{\dagger} \sC \to f^* \sC$  of $\sO_X$-bimodules. Additionally, one readily checks that it is a homomorphism of graded rings. \end{proof}

To illustrate the usefulness of the pristine pullback, we mention some results on how key constructions on $F$-singularity theory using Cartier modules commute with pristine base change. The first transformation result generalizes Blickle's and the second named author theorem on the commutativity of étale base change with functorial test modules \cite{BlickleStablerFunctorialTestModules}.

\begin{theorem} \label{thm.TransRuleTestModulesPristinePullback}
    Let $f\: Y \to X$ be a pristine morphism of $F$-finite locally noetherian schemes. Let $\sC$ be a Cartier algebra on $X$ and $\sF$ be a $\sC$-module without embedded primes. Then
    \[
    \uptau(f^{*} \sF, f^{*} \sC) = f^{*}\uptau(\sF,\sC).
    \]
    In particular, $F$-regularity and $F$-rationality are local properties in the site defined by pristine fpqc coverings.
\end{theorem}

Using the pristine pullback, this transformation rule can be proved using more direct and simpler methods than those in \cite{BlickleStablerFunctorialTestModules}. For example, it can be proved without resorting to any Galois descent. However, the proof is still rather lengthy. In a forthcoming work by the authors, we will generalize this result to the case of regular morphisms (essentially relaxing the triviality of $F_f$ to its flatness) and in particular we will give another proof there. For the sake of completeness, we decided to include a simpler version of it here as an appendix. 

To conclude, we show that common local invariants to measure singularities commute with local pristine homomorphisms. See \cite{BlickleSchwedeTuckerFSigPairs1} for further details on these prominent invariants.
 
\begin{theorem} \label{thm.TransRules}
Let $\theta \: (R,\fram,\kay) \to (S,\fran,\el)$ be a pristine local homomorphism of $F$-finite and noetherian rings and set $f = \Spec \theta$. Let $\sC$ be a Cartier $R$-algebra acting on $R$. We have the following transformation rule for the $F$-splitting ideals
\[
I_e(S,f^*\sC)=I_e(R,\sC)S.
\]
In particular, we obtain the equality between the $F$-splitting numbers $a_e(S,f^* \sC)=a_e(R,\sC)$. Moreover, the following transformation rules hold:
 \begin{enumerate}
    \item ($F$-signature) $s(S,f^*\sC) = s(R,\sC)$,
    \item ($F$-splitting primes) $\upbeta(S, f^*\sC) = \upbeta(R,\sC) S$,
    \item ($F$-splitting ratios) $r(S, f^*\sC) = r(R,\sC)$.
\end{enumerate}
\end{theorem}
\begin{proof}
We start with the following observation.

\begin{claim}
    $\fram S = \fran$ and so $\lambda_S(f^\ast M)  = \lambda_R(M)$ for all $R$-modules $M$.
\end{claim}
\begin{proof}[Proof of claim]
    Since the fiber $S/\m S$ of $f$ is zero-dimensional and regular by \autoref{theo.pristinecharacterizationnoetherian} (b) it is a field. Hence, $f^{-1}(\fram)=  \fran$ and $\m S = \n$. For the length equality, note that since $f$ is flat we have $\lambda_S(f^\ast M) = \lambda_R(M) \cdot \lambda_S(S/\fram S) = \lambda_R(M)$.
\end{proof}

Note that every $\psi \in f^\ast \sC_e$ is obtained using a commutative diagram
\[
\xymatrixcolsep{3pc}\xymatrix{F^e_\ast S \ar[r]^\psi  & S\\ F^e_\ast R \otimes_R S \ar[u]^{F^e_{\theta}} \ar[r]^{\varphi \otimes \id} &R \otimes_R S \ar[u]_{\mathrm{can}} }
\]
for some $\varphi \in \sC_e$. In particular, since the vertical maps are isomorphisms, we deduce that $\psi$ is surjective if and only if so is $\varphi$ (note that $f$ is faithfully flat). Thus $(f^\ast\sC)_e^{\mathrm{ns}} = f^\ast (\sC^{\mathrm{ns}}_e)$; with notation as in \cite[Definition 3.3]{BlickleSchwedeTuckerFSigPairs1}. 

Applying the claim to $M = \sC_e/\sC_e^{\mathrm{ns}}$ yields the equality between the $F$-splitting numbers $a_e(S,f^*\sC) = a_e(R,\sC)$. Thus, to see that $I_e(S,f^*\sC)\supset I_e(R,\sC)S$ is an equality, it suffices to show that their co-lenghts coincide. This holds because $\kay\to \el$ is pristine and so $\alpha(R)\coloneqq [\kay : \kay^p] = [\el:\el^p] \eqcolon \alpha(S)$ (see \cite[Proposition 3.5]{BlickleSchwedeTuckerFSigPairs1}).

Note that $\dim R = \dim S$ since $\dim S/\fram S = 0$ using \cite[Theorem 10.10]{EisenbudCommutativeAlgebraWithAView}. Thus, by \cite[Theorem 3.11]{BlickleSchwedeTuckerFSigPairs1} we conclude that $s(R, \sC) = s(S, f^\ast \sC)$.

For (b), by the above, we may assume that both $F$-splitting primes are proper. Let $f'$ be the pullback of $f$ along the closed subscheme cut out by $\upbeta(\sC,R)$.

By definition of the splitting prime $(R/\upbeta(\sC,R), \sC)$ is $F$-regular. Since $f'$ is pristine, we conclude that $(S/\upbeta(\sC, R)S, f^\ast \sC)$ is also $F$-regular. Since the minimal primes $\q_i$ of $\upbeta(\sC, R)S$ are $f^\ast \sC$-submodules (cf.\ \cite[Corollary 4.8]{SchwedeCentersOfFPurity}), we conclude that any $\q_i$ is a maximal proper $f^\ast \sC$-submodule of $S$. But $\upbeta(S,f^*\sC)$ is the only one. Hence, $\upbeta(\sC, R)S = \upbeta(f^\ast \sC, S)$. 

Formula (c) is an immediate consequence of (a) and (b).
\end{proof}

\begin{remark}
    By working with Frobenius powers in place of $F$-splitting primes in \autoref{thm.TransRules}, one gets that the minimal number of generators of $F_*^e R$ and $F^e_*S$ are the same; in particular, one has equality of Hilbert--Kunz multiplicities $e_{\mathrm{HK}}(R)=e_{\mathrm{HK}}(S)$. It should be noted that this asymptotic equality between Hilbert--Kunz multiplicities and $F$-signatures (however not for general Cartier algebras) is known to hold more generally for regular local maps \cite{YaoObservationsAboutTheFSignature}, \cite[Proposition 3.9b]{KunzOnNoetherianRingsOfCharP}, and \cite[\S3]{CarvajalSchwedeTuckerBertiniFSignature}. What we show here is that in the pristine case it is an equality on the nose between the respective numerical ranks of Frobenius pushforwards.
\end{remark}

\appendix

\section{Proof of \autoref{thm.TransRuleTestModulesPristinePullback}}

We may restrict ourselves to the affine setup and consider $f$ to be the spectrum of a homomorphism $R \to S$ of noetherian $F$-finite rings. We let $\sC$ be a Cartier algebra and $M$ be a $\sC$-module. We write $\sC_+ = \bigoplus_{e \geq 1} \sC_e$. Let us next recall some relevant notions from the theory of Cartier modules:

\begin{enumerate}
    \item Blickle showed in \cite[Proposition 2.13]{BlickleTestIdealsViaAlgebras} that the descending chain
    \[ \sC_+ M \supset \sC_+^2 M \supset \cdots \]
    stabilizes. Its stable member is denoted by $\underline{M}$ or $\sC_+^e M$ for $e \gg 0$.
    \item If $M = \underline{M}$, then $M$ is called \emph{$F$-pure}.
    \item If all associated primes of $M$ are minimal,\footnote{It is sufficient to consider minimal primes of $\underline{M}$ only.} then the \emph{test module} $\uptau(M, \sC)$ is the smallest Cartier submodule $N$ of $\underline{M}$ such that for every minimal prime $\eta$ of $M$ the inclusion $N_\eta \subset M_\eta$ is an equality. We say that $(M,\sC)$ is \emph{$F$-regular} if $\uptau(M,\sC)=M$.
\end{enumerate}

\begin{lemma}
\label{le.pristinefpure}
With notation as above, $(f^*\sC_{+}) f^\ast M = f^\ast \sC_{+} M$.
\end{lemma}
\begin{proof}
This proceeds as \cite[Lemma 6.1]{BlickleStablerFunctorialTestModules}.
\end{proof}

\begin{lemma}
\label{le.assprimespristinepullback}
With notation as above, suppose that $M$ has no embedded primes. Then $\Ass f^\ast M = f^{-1}(\Ass M)$ consists of the minimal primes lying over the minimal primes of $M$.
\end{lemma}
\begin{proof}
We may assume that $M$ is $F$-pure; see \cite[Lemma 1.9]{BlickleStablerFunctorialTestModules}. If $I = \Ann M$, then $I$ is a radical ideal (\cite[Proposition 2.21]{BlickleTestIdealsViaAlgebras}). Making a base change, we may assume that $R$ is reduced by \autoref{prepristineproperties} (a). By \autoref{le.pristinefpure}, $f^\ast M$ is also $F$-pure. Since $S$ is a flat $R$-module, and using \cite[\href{https://stacks.math.columbia.edu/tag/0312}{Tag 0312}]{stacks-project}, we have 
\[\Ass_S f^\ast M = \bigcup_{\p \in \Ass_R M} \Ass_S S/\p S 
\] and $\Ass S/\p S$ consists of the primes lying over $\p$. Since $\p$ is minimal, $\Ass_S f^\ast M$ is, by going down for flat maps (\cite[\href{https://stacks.math.columbia.edu/tag/00HS}{Tag 00HS}]{stacks-project}), the set of minimal primes lying above the minimal primes of $M$.
\end{proof}

\begin{theorem}
\label{theo.Fregularpristinepullback}
With notation as above, suppose that $M$ has no embedded primes.
If $(M,\sC)$ is $F$-regular, then so is $(f^\ast M, f^*\sC)$. The converse is true if $f$ is surjective.
\end{theorem}
\begin{proof}
If $f$ is surjective and $f^\ast M$ is $F$-regular, then $M$ is $F$-pure by \autoref{le.pristinefpure}. Hence, we may assume that $M$ is $F$-pure. With this in mind, we may base change to $R/\Ann_R(M)$ using \autoref{prepristineproperties} (a). We will thus assume that $R$ is reduced and $\Supp M = \Spec R$.

If $f$ is surjective, then we may argue just as in \cite[Theorem 6.5]{BlickleStablerFunctorialTestModules} to conclude that $f^\ast M$ is $F$-regular only if $M$ is $F$-regular. Conversely, since $R$ is excellent and reduced, the regular locus is dense and open. Hence, by prime avoidance, we find $c \in R$ such that $R_c$ is regular and $c$ is not contained in any minimal prime of $R$. If $f^\ast M_c$ is $F$-regular, then 
\[
\uptau(f^\ast M, f^\ast \sC) = f^\ast \sC_+ c f^\ast M = f^\ast  \sC_+ c M = f^\ast  \uptau(M, \sC) = f^\ast M,
\]
where we use \cite[Theorem 3.11]{BlickleTestIdealsViaAlgebras} and \cite[Proposition 2.9]{BlickleStablerFunctorialTestModules} for the first and penultimate equality and \autoref{le.pristinefpure} for the second equality. Thus, we may replace $R$ with $R_c$ and can assume that $R$ is regular. Moreover, we can write $R = R_1 \times \ldots \times R_\ell$, where the $R_i$ are regular domains. By \cite[Lemma 6.13]{StablerVfiltrations}, we can check for $F$-regularity on each $R_i$ separately and therefore may assume that $R$ is a (regular) domain.

By \autoref{le.assprimespristinepullback}, the associated primes of $f^\ast M$ are the minimal primes of $S$ lying over $0 \in \Spec R$. As $f$ is pristine, the generic fiber of $f$ consists precisely of all minimal primes $\eta_1, \ldots, \eta_n$ of $S$ by \autoref{theo.pristinecharacterizationnoetherian} (c).

Let $N \subset f^\ast M$ be a Cartier-submodule such that $N_{\eta_i} = f^\ast  M_{\eta_i}$ for all $\eta_i$. Write $K$ for the fraction field of $R$. Note that $S \otimes_R K = \kappa(\eta_1) \times \cdots \times \kappa(\eta_n)$ by the above. As the restriction $N'$ of $N$ to the generic fiber is the direct sum of the $N_{\eta_i}$ and 
\[
N' = N \otimes_R K = N[(R \setminus 0)^{-1})] = f^\ast M  \otimes_R K,
\]
we find $0 \neq a \in R$ such that $N_a = f^\ast M_a$. In particular, there is $s \gg 0$ such that $a^s f^\ast  M \subset N$.
On the other hand, 
\[f^\ast\sC_+ a^s f^\ast M = f^\ast  \sC_+ a^s M = f^\ast  \uptau(M, \sC) = f^\ast  M ,
\] where we use \autoref{le.pristinefpure} for the first equality and \cite[Theorem 3.11]{BlickleTestIdealsViaAlgebras} for the second. As $N$ is a Cartier module, we have $f^\ast \sC_+ a^s f^\ast  M \subset N$. Hence $N = f^\ast M$ and $f^\ast M$ is $F$-regular.
\end{proof}

As a corollary, we get \autoref{thm.TransRuleTestModulesPristinePullback}.

\begin{proof}[Proof of \autoref{thm.TransRuleTestModulesPristinePullback}]
By the exactness of $f^\ast$, we have $f^\ast \uptau(M, \sC) \subset f^\ast M$. We want to argue that $\uptau(f^\ast M, f^\ast \sC) \subset f^\ast \uptau(M, \sC)$. 
 By definition, for any minimal prime $\eta$ of $M$, we have $\uptau(M, \sC)_\eta = M_\eta$. Applying $f^\ast$, using \autoref{le.assprimespristinepullback} and the fact that we can pass from $f^{-1}(\eta)$ to any prime in the fiber by a localization since $f$ is pristine (\autoref{theo.pristinecharacterizationnoetherian} (c)), we have
\[ f^\ast(\uptau(M,\sC)))_\nu = f^\ast M_\nu\]
for all minimal primes $\nu$ of $f^\ast M$. By definition, $\uptau(f^\ast M, f^\ast \sC)$ is the smallest submodule of $f^\ast M$ for which the inclusions
\[ 
\uptau(f^\ast M, f^\ast \sC)_\nu \subset f^\ast M_\nu \] 
are equalities, so that $\uptau(f^\ast M, f^\ast \sC) \subset f^\ast \uptau(M, \sC)$. Since $f^\ast \uptau(M, \sC)$ is $F$-regular by \autoref{theo.Fregularpristinepullback}, we conclude that the equality holds.
\end{proof}

\bibliographystyle{skalpha}
\bibliography{MainBib}
\end{document}